\documentclass[reqno,11pt]{amsart}

\usepackage{amssymb,amsmath,amsthm,amsfonts,mathtools,bbm,mathrsfs,enumitem,graphics,float,multirow}
\usepackage[normalem]{ulem}
\usepackage{hyperref}
\usepackage{comment}
\usepackage{hyperref}
\hypersetup{
    colorlinks,
    linkcolor={blue!80!black},
    citecolor={red!80!black},
    urlcolor={red!80!black}
}
\usepackage{xargs}                      
\usepackage{xcolor}  
\usepackage[colorinlistoftodos,prependcaption,textsize=tiny]{todonotes}
\newcommandx{\unsure}[2][1=]{\todo[linecolor=red,backgroundcolor=red!25,bordercolor=red,#1]{#2}}
\newcommandx{\change}[2][1=]{\todo[linecolor=blue,backgroundcolor=orange!25,bordercolor=blue,#1]{#2}}
\newcommandx{\info}[2][1=]{\todo[linecolor=OliveGreen,backgroundcolor=OliveGreen!25,bordercolor=OliveGreen,#1]{#2}}

\usepackage[top=1.6in,bottom=1.5in,left=1.3in,right=1.3in]{geometry}
\allowdisplaybreaks
\usepackage{cellspace}
\newcommand\redout{\bgroup\markoverwith{\textcolor{red}{\rule[0.5ex]{2pt}{0.8pt}}}\ULon}

\newtheorem{theorem}{Theorem}
\newtheorem{lemma}[theorem]{Lemma}
\newtheorem{claim}[theorem]{Claim}

\usepackage{tikz}
\usetikzlibrary{calc}

\tikzset{font= }

\newcommand\nc\newcommand
\nc\bfa{{\boldsymbol a}}\nc\bfA{{\boldsymbol A}}\nc\cA{{\mathscr A}}
\nc\bfb{{\boldsymbol b}}\nc\bfB{{\boldsymbol B}}\nc\cB{{\mathscr B}}
\nc\bfc{{\boldsymbol c}}\nc\bfC{{\boldsymbol C}}\nc\cC{{\mathscr C}}
\nc\bfd{{\boldsymbol d}}\nc\bfD{{\boldsymbol D}}\nc\cD{{\mathscr D}}
\nc\bfe{{\boldsymbol e}}\nc\bfE{{\boldsymbol E}}\nc\cE{{\mathscr E}}
\nc\bff{{\boldsymbol f}}\nc\bfF{{\boldsymbol F}}\nc\cF{{\mathscr F}}
\nc\bfg{{\boldsymbol g}}\nc\bfG{{\boldsymbol G}}\nc\cG{{\mathscr G}}
\nc\bfh{{\boldsymbol h}}\nc\bfH{{\boldsymbol H}}\nc\cH{{\mathscr H}}\nc\fH{{\mathfrak H}}
\nc\bfi{{\boldsymbol i}}\nc\bfI{{\boldsymbol I}}\nc\cI{{\mathcal I}}
\nc\bfj{{\boldsymbol j}}\nc\bfJ{{\boldsymbol J}}\nc\cJ{{\mathscr J}}
\nc\bfk{{\boldsymbol k}}\nc\bfK{{\boldsymbol K}}\nc\cK{{\mathscr K}}
\nc\bfl{{\boldsymbol l}}\nc\bfL{{\boldsymbol L}}\nc\cL{{\mathscr L}}
\nc\bfm{{\boldsymbol m}}\nc\bfM{{\boldsymbol M}}\nc\cM{{\mathscr M}}
\nc\bfn{{\boldsymbol n}}\nc\bfN{{\boldsymbol N}}\nc\sN{{\mathscr N}}
\nc\bfo{{\boldsymbol o}}\nc\bfO{{\boldsymbol O}}\nc\cO{{\mathscr O}}
\nc\bfp{{\boldsymbol p}}\nc\bfP{{\boldsymbol P}}\nc\cP{{\mathscr P}}
\nc\bfq{{\boldsymbol q}}\nc\bfQ{{\boldsymbol Q}}\nc\cQ{{\mathscr Q}}
\nc\bfr{{\boldsymbol r}}\nc\bfR{{\boldsymbol R}}\nc\cR{{\mathscr R}}
\nc\bfs{{\boldsymbol s}}\nc\bfS{{\boldsymbol S}}\nc\cS{{\mathscr S}}
\nc\bft{{\boldsymbol t}}\nc\bfT{{\boldsymbol T}}\nc\cT{{\mathscr T}}
\nc\bfu{{\boldsymbol u}}\nc\bfU{{\boldsymbol U}}\nc\cU{{\mathscr U}}
\nc\bfv{{\boldsymbol v}}\nc\bfV{{\boldsymbol V}}\nc\cV{{\mathscr V}}
\nc\bfw{{\boldsymbol w}}\nc\bfW{{\boldsymbol W}}\nc\cW{{\mathscr W}}
\nc\bfx{{\boldsymbol x}}\nc\bfX{{\boldsymbol X}}\nc\cX{{\mathscr X}}
\nc\bfy{{\boldsymbol y}}\nc\bfY{{\boldsymbol Y}}\nc\cY{{\mathscr Y}}
\nc\bfz{{\boldsymbol z}}\nc\bfZ{{\boldsymbol Z}}\nc\cZ{{\mathscr Z}}
\nc\pp{\mathbb{P}}
\nc\ee{\mathbb{E} }

\renewcommand{\leq}{\leqslant}

\renewcommand{\geq}{\geqslant}

\nc{\Cay}{{\sf Cay}}

\nc{\ff}{{\mathbb F}}

\newcommand\remove[1]{}

\title[]{Maximum Percolation Time on the q-ary hypercube}
\author[]{Fengxing Zhu}\thanks{Institute for Systems Research and Department of ECE, University of Maryland, College Park, MD 20742, USA, fengxing@terpmail.umd.edu. Supported in part by NSF grant CCF 2330909.}

\begin{document}
\begin{abstract}
We consider the $2$-neighbor bootstrap percolation process on the $n$-dimensional $q$-ary hypercube with vertex set $V=\{0,1,\dots,q-1\}^n$ and edges connecting the pairs at Hamming distance $1$. We extend the main theorem of Przykucki (2012) about the maximum percolation time with threshold $r=2$ on the binary hypercube to the $q$-ary case, finding the exact value of this time for all $q \geq 3$. 
\end{abstract}
\maketitle

\section{Introduction}
The process of $r$-neighbor bootstrap percolation on an undirected graph $G(V,E)$ was introduced by Chalupa, Leith, and Reich \cite{Chalupa_1979}. In this process, each vertex is either infected or healthy, and once a vertex becomes infected, it remains infected forever. Let $A_0$ denote the set of initially infected vertices and let $A_i$ represent the set of infected vertices up to step $i$. The bootstrap percolation process evolves in discrete steps as follows: for $i>0$,
$$A_i = A_{i-1} \cup \{v \in V:|N_v \cap A_{i-1}| \geq r \}.$$
The process terminates, when for some $i$, $A_i=A_{i-1}$ or $A_i=V$ (percolation occurs).
Here $N_v=N_v(G)=\{u\in V: uv\in E\}$ is the neighborhood of $v$ in $G$. The parameter $r$ is called the {\em infection threshold}.

Bootstrap percolation can be studied in either a deterministic or random setting, depending on how the initial set $A_0$ is chosen. In the random setting, $A_0$ is generated by independently infecting each vertex $v\in V$ with probability $p$.
The central question in the problem of random bootstrap percolation is to determine the critical probability, defined as
   $$
   p_c(G,r)=\sup\{p \in (0,1): \mathbb{P}_p(A_0 \; \text{percolates on} \: G) \leq \frac{1}{2} \}.
   $$

 Much research has been dedicated to examining the critical probability on $d$-dimensional grid graphs $[n]^d$. Notable works include \cite{Aizenman_1988,Cirillo,CERF200269,Holroyd2002SharpMT},
with the work in \cite{Balogh} establishing a sharp estimate for $p_c([n]^d,r)$, $2 \leq r \leq d$.

In addition to grid graphs the critical probability on the $n$-dimensional binary hypercube $Q_n$ has been investigated for various infection thresholds. In \cite{balogh_bollobas_2006}, the authors obtained a tight estimate for the critical probability $p_c(Q_n,2)$ up to a constant factor, with a sharper estimate later provided in \cite{balogh_bollobas_morris_2010}. In addition to studying the critical probability when the infection threshold is a constant, \cite{balogh_bollobas_morris_2009} derived a sharp estimate for $p_c(Q_n,\frac{n}{2})$ and the exact second-order term. More recently, the main result from \cite{balogh_bollobas_2006} on the binary Hamming cube was extended to the $q$-ary Hamming cube in \cite{kang2024bootstarp}.

In the deterministic setting, one of the main problems is determining the size of the smallest "contagious set," which is the initially infected set that results in full infection under the $r$-process. Another related question is identifying the maximum number of steps required for the process to reach percolation, given that $A_0$ is contagious. Let $m(G,r)$ denote the size of a minimum contagious set and let $M(G,r)$ denote the maximum percolation time on a graph $G$ with the infection threshold $r$. For the grid graph, \cite{Benevides} found the exact first order term of $M([n]^2,2)$ and \cite{Fabricio2} determined the exact maximum percolation time on $[n]^2$ with an infection threshold of $2$ for all contagious sets of minimum size. Additionally, the authors in \cite{dukes2023extremalboundsthreeneighbourbootstrap}
found the exact value of $m([a_1] \times [a_2] \times [a_3],3)$ for all but small values of 
the dimensions $a_i$ and obtained a complete characterization of $m([n]^2,3)$ by extending the work in \cite{Fabricio}. For the hypercube, \cite{MORRISON201861} obtained a tight estimate of $m(Q_n,r)$ by using the linear algebraic technique developed in \cite{Balough2012}. In addition, \cite{Michal} derived the exact value of $M(Q_n,2)$ and \cite{Ivailo} gave an estimate for $M(Q_n,r \geq 3)$.

In this paper we address a combinatorial question in bootstrap percolation, namely, the maximum percolation time over all contagious sets on $q$-ary hypercubes with the infection threshold $r=2$. More specifically we extend the main theorem in \cite{Michal} about the maximum percolation time with threshold $r=2$ on the binary hypercubes to the $q$-ary case.

Let us define the $n$-dimensional $q$-ary hypercube, denoted by $Q_{n,q}$ , as $V(Q_{n,q})=\{0,1,\dots,q-1\}^n$ and $E(Q_{n,q})=\{xy: x,y \in \{0,1,\dots,q-1\}^n, |\{i: x_i \neq y_i\}|=1\}$. 

In order to simplify the notation, instead of writing $M(Q_{n,q},2)$, we write $M_q(n)$ to denote the maximum percolation time on the $n$-dimensional $q$-ary hypercube under $2$-neighbor bootstrap percolation. 

\begin{theorem} \label{Theorem1}
For $q \geq 3$, $M_q(0)=0$, $M_q(1)=1$ and $M_q(2)=3$. For larger $n$, 
\begin{align*}
M_{3}(n)
&=
\begin{cases}
  \frac{n^2}{3}+\frac{2n}{3} & \text{if $n=3,6,9,\cdots,$}\\ 
  \frac{n^2}{3}+\frac{2n}{3}& \text{if $n=4,7,10, \cdots,$} \\
   \frac{n^2}{3}+\frac{2n}{3}+\frac{1}{3}& \text{if $n=5,8,11,\cdots.$}\\ 
\end{cases}
\end{align*}

For $q \geq 4$,
\begin{align*}
M_{q}(n)
&=
\begin{cases}
  \frac{n^2}{3}+n & \text{if $n=3,6,9,\cdots,$}\\
  \frac{n^2}{3}+n-\frac{1}{3}& \text{if $n=4,7,10, \cdots,$} \\
   \frac{n^2}{3}+n-\frac{1}{3}& \text{if $n=5,8,11,\cdots.$}\\
\end{cases}
\end{align*}

\end{theorem}

\section{Preliminary Results}
We adopt the notation introduced in \cite{balogh_bollobas_2006} and \cite{Michal}. Specifically, we denote by $Q_l$ any of the $\binom{n}{l}q^{n-l}$ subcubes of dimension $l$ in $Q_{n,q}$. For $x=(x_i)_{i=1}^n \in \{0,1,\dots,q-1,*\}^n$, let $Q^{x}$ be the subcube $\{z=(z_i)_{i=1}^n \in \{0,1,\dots,q-1\}^n : z_i=x_i \; \text{if} \; x_i \neq *\}$. 

We will write $x=x_1x_2\cdots x_n$ and $x=(x_1,x_2.\dots,x_n)$ interchangeably. For convenience and clarity we will write $x=\underbrace{a\cdots a}_{k}\underbrace{b\cdots b}_{l}$ as $[a]^k[b]^l$.

Define the distance between two coordinates
\begin{align*}
d(x_i,y_i)
&=
\begin{cases}
  1& \text{if $x_i \neq y_i$, $x_i \neq *$ and $y_i \neq *$}, \\
   0 & \text{if either $x_i=*$ or $y_i=*$},\\
  0 & \text{if $x_i=y_i$}.\\ 
\end{cases}
\end{align*}

Using the above definition we can define the distance between two subcubes $Q^x$ and $Q^y$ in $Q_n$ as 
$$d(Q^x,Q^y)=\sum_{i=1}^n d(x_i,y_i),$$
where $x$ and $y$ represent subcubes $Q^x$ and $Q^y$. 

Let $\langle A_0 \rangle= \cup_{i=0}^{\infty} A_t$ be the set of all eventually infected vertices under the 2-neighbor process when $A_0$ is the set of initially infected vertices. 

A set $S \subset V(Q_{q,n})$ is closed if every vertex $x \in V(Q_{n,q}) \backslash S$ has at most one neighbor in $S$. Note that every subcube of a hypercube is closed. 

Let us prove some preliminary results by following \cite{balogh_bollobas_2006}. Before stating the first lemma we need to give a definition.

For $x$ and $y \in \{0,1,\dots,q-1,*\}^n$, define 
$$x \vee y =z =(z_i)_{i=1}^n,$$
where 
\begin{align*}
(z_i)=
\begin{cases}
  x_i& \text{if $x_i =y_i$,} \\
   * & \text{otherwise.}\\
\end{cases}
\end{align*}
For $x \in \{0,1,\dots,q-1,*\}^n$ we define the dimension of $x$ as $\text{dim}(x)=|\{i:x_i=*\}|$.

If the dimension of a subcube $Q^x$ is $l$ we will sometimes write it as $Q_l^x$. Similarly if we only know that the cube has dimension at most $l$ we may write it as $Q_{\leq l}^x$.

The following three lemmas were proven for the binary hypercube in \cite{balogh_bollobas_2006}. For the non-binary case the proofs follow similarly and we give the proofs for completeness.
\begin{lemma} \label{lemma:closed}
    If the hypercube $Q_{n,q}$ contains a subcube $Q_l$ with $S \subset Q_l$, then $\langle S \rangle\subset Q_l$.
\end{lemma}
\begin{proof}
Let $S \subset Q_l \subset Q_n$, where $Q_l$ is a subcube of $Q_{n,q}$. Note that $\langle S \rangle$ is the intersection of all closed sets containing $S$. 

Since the subcube $Q_l$ is a closed set containing $S$, it contains the intersection of all closed sets containing $S$ and thus it contains $\langle S \rangle$. 
\end{proof}

\begin{lemma} \label{lemma:S}
Let $x$ and $y $ be vectors $\in \{0,1,\cdots,q-1,*\}^n$ with $z=x \vee y$ and $d(x,y) \leq 2$. If $S \subset Q_{l}^x \cup Q_k^y$, then $\langle S \rangle \subset Q_{\leq l+k+2}^z$. 
\end{lemma}
\begin{proof}
Since $S \subset Q_{l}^x \cup Q_k^y \subset Q^{z}$ and the subcube $Q^z$ is closed, we have $\langle S \rangle \subset Q^z$. Since the number of $*$ in the coordinates of $z$ is at most $l+k+d(x,y)$, the dimension of the subcube $Q^z$ is $\leq l+k+2$.
\end{proof}

\begin{lemma} \label{lemma:Q^x Q^y}
For $x$ and $y \in \{0,1,\dots,q-1,*\}^n$ with $d(x,y) \leq 2$, we have 
$$\langle Q^x \cup Q^y \rangle = Q^{x \vee y}.$$
\end{lemma}
\begin{proof}
 Since the subcube $Q^{x \vee y}$  is closed and $Q^x \cup Q^y \subset Q^{x \vee y},$ we have 
 $\langle Q^x \cup Q^y \rangle \subset Q^{x \vee y}.$ by Lemma \ref{lemma:S}. What remains is to show $ Q^{x \vee y} \subset \langle Q^x \cup Q^y \rangle$. 
 
 W.l.o.g we can assume that 
 $$x=[\ast]^k[0]^{n-k}$$
 $$y=[\ast]^j[0]^{k-j}[\ast]^{l-j}[0]^{n-k-l+j-2}ab,$$
where $a,b\in \{1,2,...,q-1\}$ and $0 \leq j \leq k \leq n-l+j-2$. 

We will use the following claim several times in the proof.
\begin{claim} \label{claim:A be a subset}
Let $A$ be a subset of $V(Q_{n,q})$. If a vertex u has at least 2 neighbors in $A$, then $u \in \langle  A \rangle$. 
\end{claim}
\begin{proof}
Let $u=u_1u_2\cdots u_n$. Since $u$ has at least 2 neighbors in $A$, w.l.o.g the two neighbors $y$ and $z$ of $u$ are $y=au_2u_3\cdots u_n$ and $z=bu_2\cdots u_n$ where $a \notin \{b,u_1\}$ and $b \notin \{a,u_1\}$ or $y=au_{2}u_{3}\cdots u_n$ and $z=u_{1}bu_{3}\cdots u_{n}$ where $a \neq u_1$ and $b \neq u_2$. 

If $y=au_2u_3\cdots u_n$ and $z=bu_1u_2\cdots u_n$, then $\langle y \cup z \rangle =\ast u_2u_3 \cdots u_n$ and thus $u \in \langle A \rangle$. 

If $y=au_{2}u_{3}\cdots u_n$ and $z=u_{1}bu_{3}\cdots u_{n}$, then $\langle y \cup z \rangle =\ast \ast u_3u_4\cdots u_n$ and thus $u \in \langle A \rangle$.  \end{proof}

For $ 0 \leq  i \leq k+l-2j$, let $z_i$ be 
$$z_i=[\ast]^{i+j}0^{n-i-j-2}\ast \ast.$$

We will prove by induction on $i$ that $Q^{z_i} \subset \langle Q^x \cup Q^y \rangle$ for $ 0 \leq i \leq k+l-2j$, which is sufficient to prove this lemma since $z_{k+l-2j}=x \vee y$.

Consider a vertex $u$ in $Q^{z_0} \backslash (Q^x \cup Q^y)$. If $u=u_1 u_2 \cdots u_{n-2}a0$ or $u=u_1 u_2\cdots u_{n-2}0b$ then it has one neighbor in $Q^{x}$ and one neighbor in $Q^{y}$. Then by Claim \ref{claim:A be a subset}, $u \in \langle Q^x \cup Q^y \rangle$. 

For a vertex $u=u_1u_2\cdots u_n$ we will write $u^{ab}=u_1u_2\cdots u_{n-2}ab$. Now if $u=u_1 u_2 \cdots u_{n-2}aj$ with $j \notin \{0,b\}$ then it has two neighbors $u^{ab}=u_1u_2 \cdots u_{n-2}ab$ and $u^{a0}=u_1u_2\cdots u_{n-2}a0$ and thus $u \in \langle Q^x \cup Q^y \rangle$. Similarly, a vertex $u=u_1u_2 \cdots u_{n-2}jb$ for $j \notin \{0,a\}$ is also in $\langle Q^x \cup Q^y \rangle$. Finally if $u=u_1u_2 \cdots u_{n-2} ij$ where $i \notin \{0,a\}$ and $j \notin \{0,b\}$, it has two neighbors in $Q^x$ and $Q^y$ and thus $u \in \langle Q^x \cup Q^y \rangle$. Indeed, one of them is $u_1u_2 \cdots u_{n-2} u_{n-1}b$ and another one is $u=u_1u_2 \cdots u_{n-2} au_{n}$.

Now assume that for $ i \geq 0$, $Q^{z_i} \subset \langle Q^x \cup Q^y \rangle$ and our goal is to show that $Q^{z_{i+1}} \subset \langle Q^x \cup Q^y \rangle$. 

Let a vertex $u \in Q^{z_{i+1}} \backslash (Q^{z_i} \cup Q^{x} \cup Q^{y})$. Thus $u$ has a neighbor in $Q_{z_i}$ and thus we would like to find another neighbor of $u$ in $\langle Q^x \cup Q^y \rangle$.

First let us consider the case when $ i \leq k-j-1$. Since $u \notin Q^x$, the last two digits of of $u$ cannot be both $0$. If $u=u_1u_2 \cdots u_{n-2}a0$, then it has a neighbor is $Q^x$ and by Lemma \ref{claim:A be a subset} $u \in \langle Q^x \cup Q^y \rangle$. Similarly, $u^{0b} \in \langle Q^x \cup Q^y \rangle$ and it is easy to see that $u^{\ast \ast} \in \langle Q^x \cup Q^y \rangle$.  

Now assume that $i \geq k-j$. Let $u=u_1u_2 \cdots u_{n-2}ab$. Let $r$ be the smallest value in the range $ 0 \leq r \leq k-j$ such that $u_{j+r+1}=\cdots=u_k=0$. We apply induction on $r$ to prove that $u \in \langle Q^x \cup Q^y \rangle$. Let $r=0$, then $u \in Q^y \subset \langle Q^x \cup Q^y \rangle$. Let us turn to the induction step. Thus $1 \leq r \leq k-j $,$u_{j+r}=1$ and $u_{j+r+1}=\cdots=u_k=0$. Since $u \notin Q^{z_i}$, we have $u_{i+j+1} \neq 0$, and so switching $u_{i+j+1}$ to $0$ we obtain a neighbor of $u$ in $Q^{z_i}$. Since $r \leq k-j \leq i$, then $j+r \neq i+j+1$. Therefore if we change the ($j+r$)th digit of $u$ to $0$ then by the induction hypothesis we get a vertex in $\langle Q^x \cup Q^y \rangle$. Thus by Claim $\ref{claim:A be a subset}$ $u \in \langle Q^x \cup Q^y \rangle$. 

Then it is easy to see that $u^{ij} \in \langle Q^x \cup Q^y \rangle$ for all $i,j \in \{0,1,\cdots,q-1\}$. 
\end{proof}

The following three lemmas were proven for the binary hypercube in \cite{balogh_bollobas_2006} and for the non-binary case the proofs are identical, by applying Lemmas (\ref{lemma:closed}),(\ref{lemma:S}), and (\ref{lemma:Q^x Q^y}). Therefore we ommit the proofs.

\begin{lemma} \label{lemma:The only subsets}
(i)The only subsets of the hypercube that are closed under $2$-neighbor percolation are those which are a union of disjoint subcubes that are at distance at least $3$ from each other.   

(ii) If a subcube $Q_l= \langle S \rangle $, then $|S| \geq \frac{l}{2}+1$.
\end{lemma}

Before stating the next lemma, we need to give a definition. For a given bootstrap percolation process on $Q_{n,q}$ with an initially infected set of vertices $A \subset Q_{n,q}$, a subcube $Q_l \subset Q_{n,q}$ is said to be internally spanned if $\langle A \cap Q_l \rangle =Q_l$.

\begin{lemma} \label{lemma: internally spanned}
If the initially infected set percolates on the hypercube $Q_{n,q}$, then for each $k \leq n$ there is an integer $l$ such that $k \leq l \leq 2k$ and the hypercube $Q_{n,q}$ contains an internally spanned cube of dimension $l$. 
\end{lemma}

\begin{lemma} \label{lemma:nested sequence}
Let $A \subset Q_{n,q}$ be such that $\langle A \rangle =Q_{n,q}$. Then there is a nested sequence $Q_0=Q_{i_1}^{x_{i_1}} \subset Q_{i_2}^{x_{i_2}} \subset \cdots \subset Q_{i_t}^{x_{i_t}}=Q_{n,q}$, of internally spanned subcubes with respect to $A$, where $2i_j+2 \geq i_{j+1}$ for all $j$, $ 0 \leq j \leq t-1$. Furthermore, for $j \geq 2$ each subcube $Q_{i_j}^{x_{i_j}}$ is spanned by two internally spanned cubes, namely by $Q_{i_{j-1}}^{x_{i_{j-1}}}$ and a subcube $Q_{m_{j-1}}$ with $m_{j-1} \leq i_{j-1}$ which is not a member of the sequence. 
\end{lemma}

Let $A_0$ be the set of initially infected vertices in a graph $G$ and let $$T_{G}(A_0)= \min \{t:A_t=\langle A_0 \rangle\}.$$ 

For the $n$-dimensional $q$-ary hypercube its maximum percolation time is defined to be  
$$M_q(n):= \max_{A: \langle A \rangle =Q_{n,q}} T_{Q_{n,q}}(A).$$

\section{Maximum percolation time}
In this section, we prove Theorem~\ref{Theorem1} by drawing on ideas from \cite{Michal}, beginning with a lemma on the monotonicity of the maximum percolation time. We then define several new operations, which are slight modifications of those introduced in \cite{Michal}. With those defined operations, we then establish six technical lemmas (which are similar to those established in \cite{Michal} but more technical) that characterize the spread of infection on $Q_{n,q}$ based on the set of initially infected vertices. Using these lemmas along with some preliminary results, we derive a recursive formula for the maximum percolation time, ultimately obtaining the desired closed-form expression for $M_{n,q}$.
 
\begin{lemma}
For any $n \in \mathbb{N}$, $M_{q}(n) \leq M_q(n+1)$. 
\end{lemma}

\begin{proof}
Let $A_0$ be the set of initially infected vertices such that $\langle A_0 \rangle =Q_{n,q}$ and $T_{Q_{n,q}}(A_0)=M_q(n)$. Let 
$$A_0^{'}=\{(a_1,\dots,a_n,i):(a_1,\dots,a_n) \in A_0 \; \text{and} \;i \in \{0,\dots,q-1\}\}$$
Then $\langle A_0^{'} \rangle=Q_{n+1,q}$ and $M_q(n)=T_{Q_{n,q}}(A_0)=T_{Q_{n+1}}(A_0^{'}) \leq M_{q}(n+1)$.
\end{proof}

We need to define the following operation. Let $n, d,n_1, n_2 \in \{1,2,\dots\}$ with $n \geq d+n_1+n_2$ and $a_i \in \{1,2,\dots,q-1\}$ for all $i \in [d]$. For each $x \in \{0,1,\cdots,q-1\}^n$ with $x=x_1x_2\dots x_n$, set
$$\left \lVert x \right \rVert_{n_1,n_2}^{a_1,\cdots,a_d}=\left(\sum_{i=1}^{n-d}\mathbb{I}_{\{x_i \neq 0\}} \right)\mathbb{I}_{\{\sum_{i=1}^{n_1}x_i >0\}} \mathbb{I}_{\{\sum_{i=n_1+1}^{n_1+n_2}x_i >0\}}\prod_{i=1}^d \mathbb{I}_{\{x_{n-d+i}=a_i\}}.$$

In words, $\left \lVert x \right \rVert_{n_1,n_2}^{a_1,\cdots,a_d}$ is the number of the none-zero elements in $\{x_1,\dots,x_{n-d}\}$ if the following three conditions are satisfied:

(1) There is at least one nonzero element in $\{x_1,x_2,\dots,x_{n_1}\}$,

(2) There is at least one nonzero  element in $\{x_{n_1+1},\dots,x_{n_1+n_2}\}$,

and

(3) The last $d$ digits of $x$ are exactly equal to $a_1,a_2,\dots,a_d$, i.e.,
$x_{n-d+i}=a_i$ for all $i \in [d]$. 

By the same spirit we will define 
$$||x||:=\sum_{i=1}^{n}\mathbb{I}_{\{x_i \neq 0\}},$$
$$||x||^{a_1,\dots,a_d}:=\left(\sum_{i=1}^{n-d}\mathbb{I}_{\{x_i \neq 0\}} \right)\prod_{i=1}^d \mathbb{I}_{\{x_{n-d+i}=a_i\}},$$

$$\left \lVert x \right \rVert_{n_1}^{a_1,\cdots,a_d}:=\left(\sum_{i=1}^{n-d}\mathbb{I}_{\{x_i \neq 0\}} \right)\mathbb{I}_{\{\sum_{i=1}^{n_1}x_i >0\}} \prod_{i=1}^d \mathbb{I}_{\{x_{n-d+i}=a_i\}},$$

$$\left \lVert x \right \rVert_{n_1}:=\left(\sum_{i=1}^{n}\mathbb{I}_{\{x_i \neq 0\}} \right)\mathbb{I}_{\{\sum_{i=1}^{n_1}x_i >0\}},$$

and 
$$\left \lVert x \right \rVert_{n_1,n_2}:=\left(\sum_{i=1}^{n}\mathbb{I}_{\{x_i \neq 0\}} \right)\mathbb{I}_{\{\sum_{i=1}^{n_1}x_i >0\}} \mathbb{I}_{\{\sum_{i=n_1+1}^{n_1+n_2}x_i >0\}}.$$

We now proceed to prove six technical lemmas that provide useful characterizations of the spread of infection for various configurations of initially infected vertices.

Recall that $A_0$ denotes the set of initially infected vertices on $Q_{n,q}$.

\begin{lemma} \label{lemma:ST1}
Let $k,l \in \mathbb{N}_0$, $n=k+l$. Let $S=[*]^k[0]^l$ and $T=[0]^k[*]^l$. With $A_0=S \cup T,$ then we have
$$A_t \supset \{x \in \{0,1,\dots,q-1\}^n: \left \lVert x \right \rVert \leq t+1\},$$
for all $t \in \mathbb{N}_0$.
\end{lemma}
\begin{proof}
Let $t=0$. For every vertex $x$ with $\left \lVert x \right \rVert \leq 1$, we have $x \in A_0$. 

Now assume that for $t \geq 0$
$$A_t \supset \{x \in \{0,1,\dots,q-1\}^n: \left \lVert x \right \rVert \leq t+1\},$$

Consider a vertex $x$ with $\left \lVert x \right \rVert =t+2 \geq 2$. This vertex $x$ has at least 2 neighbors $y$ with $\left \lVert y \right \rVert =t+1$. Therefore, $x \in A_{t+1}$. Thus
$$A_{t+1} \supset \{x \in \{0,1,\dots,q-1\}^n: \left \lVert x \right \rVert \leq t+2\}.$$
\end{proof}

\begin{lemma} \label{lemma:ST2}
Let $k,l \in \mathbb{N}_0$, $n=k+l+1$. Let $S=[*]^k[0]^{l+1}$ and $T=[0]^k[*]^li$, where $i \neq 0$.With $A_0=S \cup T,$ then we have

$$A_t \supset \left([0]^{k+l}* \right) \bigcup_{j=0,i} \{x \in \{0,1,\dots,q-1\}^n: 1 \leq \left \lVert x \right \rVert^{j} \leq t\} \bigcup_{j \neq 0,i} \{x \in \{0,1,\dots,q-1\}^n: 1 \leq \left \lVert x \right \rVert^{j} \leq t-1\},  $$ 
for all $t \in \mathbb{N} $.
\end{lemma}
\begin{proof}
Let $t=1$. It is easy to see every vertex in $[0]^{k+l}*$ has two neighbors in $ S \cup T$. Therefore, $[0]^{k+l}* \subset A_1$. 
Now consider $x$ such that $\left \lVert x \right \rVert^0=1$. Either $x \in S$ or there exists exactly $j$ with $ k<j \leq n-1$ such that $x_j \neq 0 $. For the latter case, $x$ has one neighbor in $S$ and one neighbor in $T$. 
For the case where $\left \lVert x \right \rVert^{i}$, the analysis is very similar. 
Therefore $$A_1 \supset \left([0]^{k+l}* \right) \bigcup_{j=0,i} \{x \in \{0,1,\dots,q-1\}^n: \left \lVert x \right \rVert^{j}=1\}. $$

Assume for $t \geq 1$, 
$$A_t \supset \bigcup_{j=0,i} \{x \in \{0,1,\dots,q-1\}^n: 1 \leq \left \lVert x \right \rVert^{j} \leq t\}.$$
Consider $\left \lVert x \right \rVert^0=t+1$. It is easy to see that $x$ has two neighbors $y$ with $\left \lVert y \right \rVert^0=t$. Therefore, $x \in A_{t+1}$. Similarly, $x \in A_{t+1}$ with $\left \lVert x \right \rVert^i=t+1$. Thus 
$$A_{t+1} \supset \bigcup_{j=0,i} \{x \in \{0,1,\dots,q-1\}^n: 1 \leq \left \lVert x \right \rVert^{j} \leq t+1\}.$$

Now let $t\geq 2$ and consider $x$ with $\left \lVert x \right \rVert^{j}=t-1$ where $j \notin \{0,i\}$. Note that $x$ has one neighbor $y_1$ in $\{y: \left \lVert y \right \rVert^{0}=t-1\}$ and one neighbor $y_2$ in $\{y: \left \lVert y \right \rVert^{i}=t-1\}$. Moreover, $y_1,y_2 \in A_{t-1}$. Therefore $x \in A_{t}$. 
\end{proof}

\begin{lemma} \label{lemma:ST3}
Let $k,l \in \mathbb{N}_0$, $n=k+l+2$. Let $S=[*]^k[0]^{l+2}$ and $T=[0]^k[*]^lab$. With $A_0=S \cup T,$ where $a,b \neq 0$, then we have
$$A_1 \supset \{[0]^{k+l}0b,[0]^{k+l}a0,[0]^{k+l}ab,[0]^{k+l}00\},$$
$$A_2 \supset \{[0]^{k+l}0*,[0]^{k+l}a*,[0]^{k+l}*0,[0]^{k+l}*b\},$$
$$A_3 \supset \{[0]^{k+l}**\},$$
and for all $t \geq 2$,

\begin{align}
 A_t & \supset \{[0]^{k+l}**\}\label{eq:0b}
 \bigcup\{x \in \{0,1,\dots,q-1\}^n:  1 \leq \left \lVert x \right \rVert^{0b} \leq t-1\}\\ 
 \label{eq:a0}
&\bigcup \{x \in \{0,1,\dots,q-1\}^n:  1 \leq \left \lVert x \right \rVert^{a0} \leq t-1\}\\ 
\label{eq:0i}
&\bigcup_{i \neq b} \{x \in \{0,1,\dots,q-1\}^n:  1 \leq \left \lVert x \right \rVert^{0i} \leq t-2\} \\  
\label{eq:jo}
&\bigcup_{j \neq a} \{x \in \{0,1,\dots,q-1\}^n:  1 \leq \left \lVert x \right \rVert^{j0} \leq t-2\} \\ 
\label{eq:ai}
&\bigcup_{i \neq 0} \{x \in \{0,1,\dots,q-1\}^n:  1 \leq \left \lVert x \right \rVert^{ai} \leq t-2\} \\ 
\label{eq:jb}
&\bigcup_{j \neq 0} \{x \in \{0,1,\dots,q-1\}^n:  1 \leq \left \lVert x \right \rVert^{jb} \leq t-2\} \\ 
\label{eq:cd}
&\bigcup_{\substack{c \neq 0,a \\ d \neq 0,b}} \{x \in \{0,1,\dots,q-1\}^n:  1 \leq \left \lVert x \right \rVert^{cd} \leq t-3\}. 
 \end{align}

\end{lemma}

\begin{proof}
Consider the vertex $[0]^{k+l}0b$ and it has one neighbor in $S$ and one neighbor in $T$. Therefore this vertex is in $A_1$. The same holds for the vertex $[0]^{k+l}a0$. Note that both $[0]^{k+l}ab$
and $[0]^{k+l}00$ are in $A_0$. 

First consider the vertex $[0]^{k+l}0j$ where $j \notin \{0,b\}$. It has two neighbors in $A_1$, namely, $[0]^{k+l}0b$ and $[0]^{k+1}00$. The same holds for the vertices in $\{[0]^{k+l}0*,[0]^{k+l}a*,[0]^{k+l}*0,[0]^{k+l}*b\}$. Now Consider the vertex $[0]^{k+l}ij$ where $i \notin \{0,a\}$ and $j \notin \{0,b\}$.This vertex has at least 2 neighbors in $A_2$. Therefore, this vertex is in $A_3$.

Now let $t=2$ and consider a vertex $x$ with $\left \lVert x \right \rVert^{0b}=1$. It has at least 2 neighbors in $A_1$. Indeed, one of the neighbors of vertex $x$ is $[0]^{k+l}0b$ and another one is in $S$ if there exists a $j$ with  $ 1 \leq j \leq k$ such that $x_j \neq 0$  or in $T$ if there exists a $j$ with $ k+1 \leq j \leq k+l$ such that $x_j \neq 0$. The analysis is the same for $x$ with $\left \lVert x \right \rVert^{a0}=1$.

Now assume that for $t \geq 2$, 
\begin{equation*}
 A_t \supset \{x \in \{0,1,\cdots,q-1\}^n:  1 \leq \left \lVert x \right \rVert^{0b} \leq t-1\}\bigcup \{x \in \{0,1,\cdots,q-1\}^n:  1 \leq \left \lVert x \right \rVert^{a0} \leq t-1\}.
\end{equation*}
Consider a vertex $x$ with $\left \lVert x \right \rVert^{0b}=t$ and it has 2 neighbors $y$ satisfying $\left \lVert y \right \rVert^{0b}=t-1$. By the hypothesis both neighbors are in $A_t$ and thus $x \in A_{t+1}$. The analysis for the case where the vertex $x$ satisfies $\left \lVert x \right \rVert^{a0}=t$ is the same. This finishes the proof of (\ref{eq:0b}) and (\ref{eq:a0}). 

Now let us prove (\ref{eq:0i})--(\ref{eq:jb}). Let $t=3$ and consider a vertex $x$ with $\left \lVert x \right \rVert^{0i}=1$. The vertex $x$ has 2 neighbors in $A_2$. Indeed one is $[0]^{k+l}0i$ and another neighbor is $y$ satisfying $\left \lVert y \right \rVert^{0b}=1$. Therefore, the vertex $x$ is in $A_3$. The analysis is the same for the vertex $x$ satisfying (\ref{eq:jo}),(\ref{eq:ai}), or (\ref{eq:jb}). Thus, for $t=3$, (\ref{eq:0i})--(\ref{eq:jb}) are satisfied. 

Let us assume for $t \geq 3$,
\begin{align*}
A_t \supset \bigcup_{i \neq b} \{x \in \{0,1,\cdots,q-1\}^n:  1 \leq \left \lVert x \right \rVert^{0i} \leq t-2\}. 
\end{align*}
Now consider a vertex $x$ satisfying $\left \lVert x \right \rVert^{0i}=t-1$. It is easy to see that $x$ has at least 2 neighbors $y$ satisfying $\left \lVert y \right \rVert^{0i}=t-2$. Therefore, $x$ is in $A_{t+1}$. Similarly, (\ref{eq:jo}),(\ref{eq:ai}), and (\ref{eq:jb}) can be proved.

Now we will prove (\ref{eq:cd}). Let $t=4$ and consider a vertex $x$ satisfying $\left \lVert x \right \rVert^{cd}=1$ with $c\notin \{0,a\}$ and $d \notin \{0,b\}$. Note that $x$ has at least two neighbors in $A_3$, of one which is $[0]^{k+l}cd$ and another is $y$ satisfying $\left \lVert y \right \rVert^{c0}=1$. 

We assume that for $t \geq 4$,
\begin{align*}
 A_t \supset \bigcup_{\substack{c \neq 0,a \\ d \neq 0,b}} \{x \in \{0,1,\dots,q-1\}^n:  1 \leq \left \lVert x \right \rVert^{cd} \leq t-3\}.    
\end{align*}

Consider a vertex $x$ satisfying $\left \lVert x \right \rVert^{cd}=t-2$ and it has at least two neighbors $y$ satisfying $\left \lVert y \right \rVert^{cd}=t-3$. Thus the vertex $x$ is in $A_{t+1}$. 
\end{proof}

\begin{lemma} \label{lemma:ST4}
Let $k,l \in \mathbb{N}_0$, $n=k+l$. Let $S=[*]^k[0]^l$ and $T=[0]^k[*]^l$. With $A_0=S \cup T,$ then we have
$$A_t \cap \{x \in \{0,1,\dots,q-1\}^n: \left \lVert x \right \rVert_{k,l} \geq t+2\}=\emptyset,$$
for all $ 0 \leq t \leq k+l-2$.

\end{lemma}

\begin{proof}
Let $t=0$ and consider a vertex $x$ with $\left \lVert x \right \rVert_{k,l}=2$. Then it is clear that $x \notin A_0$. 
Now let $t \geq 0$ and assume that 
$$A_t \cap \{x \in \{0,1,\dots,q-1\}^n: \left \lVert x \right \rVert_{k,l} \geq t+2\}=\emptyset.$$
Consider a vertex $x$ with $\left \lVert x \right \rVert_{k,l} \geq t+3$ and it has at most one neighbor $y$ satisfying $\left \lVert y \right \rVert_{k,l} < t+2$. Indeed if $\sum_{j=1}^k \mathbb{I}_{\{x_j \neq 0\}}=1$, we have $y=[0]^kx_{k+1}\cdots x_n$ and if $\sum_{j=k+1}^n \mathbb{I}_{\{x_j \neq 0\}}=1$, we have $y=x_1 \cdots x_k[0]^{l}$. Therefore, $x \notin A_{t+1}$.
\end{proof}

\begin{lemma} \label{lemma:ST5}
Let $k,l \in \mathbb{N}_0$, $n=k+l+1$. Let $S=[*]^k[0]^{l+1}$ and $T=[0]^k[*]^li$ where $i \neq 0$. With $A_0=S \cup T,$ then we have the following.

For all $0 \leq t \leq k+l-1$,
$$A_t \cap \{x \in \{0,1,\dots,q-1\}^n: \left \lVert x \right \rVert_{k}^i \geq t+1\}=\emptyset.$$

For $1 \leq t \leq k+l$ and $j \notin \{0,i\}$,
$$A_t \cap \{x \in \{0,1,\dots,q-1\}^n: \left \lVert x \right \rVert_{k}^j \geq t\}=\emptyset.$$

\end{lemma}
\begin{proof}
We will first prove that for all $ 0 \leq t \leq k+l-1$,
$$A_t \cap \{x \in \{0,1,\dots,q-1\}^n: \left \lVert x \right \rVert_{k}^j \geq t+1\}=\emptyset,$$
where $j \neq 0$.

Let $t=0$ and consider a vertex $x$ with $\left \lVert x \right \rVert_{k}^j \geq 1$ where $j \neq 0$. It is easy to see that $x \notin A_0$. 
Now let $t \geq 0$ and assume that for $j \neq 0$, 
$$A_{t} \cap \{x \in \{0,1,\dots,q-1\}^n: \left \lVert x \right \rVert_{k}^j \geq t+1\}=\emptyset.$$
Then consider a vertex $x$ with $\left \lVert x \right \rVert_{k}^j \geq t+2$ with $j \neq 0$. It has $q-1$ neighbors $y$ such that $\left \lVert y \right \rVert_{k}^j \leq t$, obtained by changing $x_n$ to one of the other $q-1$ symbols in $\{0,1,\dots,q-1\}$. By the induction hypothesis, among these $q-1$ neighbors, at most one $y=x_1x_2\cdots x_{n-1}0$ may be in $A_{t}$. 
If $\sum_{i=1}^k\mathbb{I}_{\{x_i \neq 0\}} =1$, another neighbor $z$ of $x$ satisfying $\left \lVert z \right \rVert_{k}^j=0$ has not been considered. Indeed, $z=[0]^{k}x_{k+1}\cdots x_{n-1}j$.

We claim that if $z$ exists then $z \notin A_{t}$. If such a $z$ exits, we have $\sum_{i=1}^k\mathbb{I}_{\{x_i \neq 0\}}=1$ and $\sum_{i=k+1}^{k+l}\mathbb{I}_{\{x_i \neq 0\}}\geq t+1$. Now let $z'=(z_{k+1},\dots,z_{k+l},z_1,\cdots,z_k,z_{k+l+1})$. Let $S'=[\ast]^l[0]^{k+1}$ and $T'=[0]^l[\ast]^ki$. Then we have $\left \lVert z' \right \rVert_{l}^j \geq t+1.$ By the induction hypothesis, it follows that $z' \notin A^{'}_t$, where $A^{'}_t$ is the set of infected vertices after $t$ steps by the 2 neighbor process with the initially infected set $S' \cup T'$. By symmetry, $z \notin A_t$. Therefore, $x \notin A_{t+1}$. 

Now we will try to prove for $1 \leq t \leq k+l$ and $j \notin \{0,i\}$,
\begin{align}
A_t \cap \{x \in \{0,1,\dots,q-1\}^n: \left \lVert x \right \rVert_{k}^j \geq t\}=\emptyset. \label{eq:jk>t}
\end{align}
Let $t=1$ and consider a vertex $x$ with $\{0,1,\dots,q-1\}^n: \left \lVert x \right \rVert_{k}^j \geq t\}$ where $j \notin \{0,i\}$. It is clear that this vertex $x$ has at most one neighbor in $A_0$ specifically in $S$. Therefore, $x \notin A_1$. 

Now assume that for $t \geq 1$, (\ref{eq:jk>t}) is satisfied. Now consider a vertex $x$ with $\left \lVert x \right \rVert_{k}^j \geq t+1$. Consider the neighbors $y$ of $x$ satisfying $\left \lVert x \right \rVert_{k}^j \leq t-1$. Specifically, we have $y$ satisfying $y=x_1\cdots x_{n-1}m$ for $m \neq j$. If $m=i$, we have $\left \lVert y \right \rVert_{k}^i \geq t+1$ implying that $y \notin A_t$. If $m=0$, $y$ might be in $A_t$. If $m \notin \{0,i,j\}$, we have $\left \lVert y \right \rVert_{k}^i \geq t+1$, by the induction hypothesis, implying that $y \notin A_t$. 

If $\sum_{i=1}^k\mathbb{I}_{\{x_i \neq 0\}} =1$, another neighbor $z$ of $x$ satisfying $\left \lVert z \right \rVert_{k}^j=0$ has not been considered. Indeed, $z=[0]^{k}x_{k+1}\cdots x_{n-1}j$. By the same argument, $z \notin A_t$. Therefore, $x \notin A_{t+1}$.
\end{proof}

\begin{lemma} \label{lemma:ST6}
Let $k,l \in \mathbb{N}_0$, $n=k+l+2$. Let $S=[*]^k[0]^{l+2}$ and $T=[0]^k[*]^lab$ where $a,b \neq 0$. With $A_0=S \cup T,$ then we have the following. 

For $t \geq 1$,
$$A_t \cap \{x \in \{0,1,\dots,q-1\}^n: \left \lVert x \right \rVert_k^{0b} \geq t\}=\emptyset,$$
and 
$$A_t \cap \{x \in \{0,1,\dots,q-1\}^n: \left \lVert x \right \rVert_k^{a0} \geq t\}=\emptyset.$$

For $j \notin \{b,0\}$ and $t \geq 2$,
$$A_t \cap \{x \in \{0,1,\dots,q-1\}^n: \left \lVert x \right \rVert_k^{0j} \geq t-1\}=\emptyset.$$

For $i \notin \{a,0\}$ and $t \geq 2$,
$$A_t \cap \{x \in \{0,1,\dots,q-1\}^n: \left \lVert x \right \rVert_k^{i0} \geq t-1\}=\emptyset.$$

For $j \neq 0$ and $t \geq 2$,
$$A_t \cap \{x \in \{0,1,\dots,q-1\}^n: \left \lVert x \right \rVert_k^{aj} \geq t-1\}=\emptyset.$$

For $i \neq 0$ and $t \geq 2$,
$$A_t \cap \{x \in \{0,1,\dots,q-1\}^n: \left \lVert x \right \rVert_k^{ib} \geq t-1\}=\emptyset.$$

For $i \notin \{a,0\}$, $j \notin \{b,0\}$, and $t \geq 3$, 
$$A_t \cap \{x \in \{0,1,\dots,q-1\}^n: \left \lVert x \right \rVert_k^{ij} \geq t-2\}=\emptyset.$$

\end{lemma}

\begin{proof}
Let $t=1$ and consider a vertex $x$ with $ \left \lVert x \right \rVert_k^{0b}\geq 1$ where $b \neq 0$. It is easy to see that this vertex $x$ has at most one neighbor in $A_0$ specifically in $S$. Thus $x \notin A_1$. Similarly, the vertex $y$ with $\left \lVert y \right \rVert_k^{a0} \geq 1$ is not in $A_1$. 

Now let $t=2$ and consider a vertex $x$ with $ \left \lVert x \right \rVert_k^{0j}\geq 1$ where $j \notin \{0,b\}$. Note that the vertex $x$ has at most one neighbor in $A_0$ specifically in $S$. Thus $x \notin A_1$. Similarly, the vertex $y$ with $\left \lVert y \right \rVert_k^{i0} \geq 1$, where $i \notin \{a,0\}$, is not in $A_1$. 

Again let $t=2$ and consider a vertex $x$ with $ \left \lVert x \right \rVert_k^{aj}\geq 1$ where $j \neq 0$. Note that the vertex $x$ has no neighbor in $A_0$ unless $\sum_{i=1}^k\mathbb{I}_{\{x_i >0\}}=1$ and $x_{n-1}=a $ and $x_n=b$. It is easy to see that even for a such $x$ it has just one neighbor in $A_0$. Thus $x \notin A_1$. Similarly, the vertex $y$ with $\left \lVert y \right \rVert_k^{ib} \geq 1$ where $i \neq 0$, is not in $A_1$. 

Now let $t=3$ and consider a vertex $x$ with $ \left \lVert x \right \rVert_k^{ij}\geq 1$ where $i \neq a,0$ and $j \notin \{0,b\}$. Note that the vertex $x$ has no neighbor in $A_0$. Thus $x \notin A_1$.

The neighbors $y$ of a vertex $x$ with $\left \lVert x \right \rVert_k^{0j} \geq 1$ satisfying $\left \lVert y \right \rVert_k^{0j}=0$ will be considered since the other neighbors of $x$ will not be in $A_1$. Specifically, we only need to analyze $y=x_1\cdots x_{n-2}00$, $y=x_1\cdots x_{n-2}0i$ with $i \notin \{0,j\}$, $y=x_1\cdots x_{n-2}ij$ with $i \neq 0$ and $y=[0]^kx_{k+1}\cdots x_{n-2}0j$ if $\sum_{i=1}^k\mathbb{I}_{\{x_i \neq 0\}}=1$. The only neighbors $y$ of $x$ which are possibly in $A_1$ are the ones satisfying $y=x_1\cdots x_{n-2}00$ and $y=[0]^kx_{k+1}\cdots x_{n-2}0j$. However, it is clear that the vertex $y$ with $y=[0]^kx_{k+1}\cdots x_{n-2}0j$ is not in $A_1$ since it has at most one neighbor in $A_0$ specifically in $S$. Therefore, the vertex $x \notin A_2$. Similarly, the vertex $x$ with $\left \lVert x \right \rVert_k^{i0} \geq 1$ where $i \notin \{a,0\}$, is not in $A_2$. 

Now we will show that the vertex $x$ with $\left \lVert x \right \rVert_k^{aj} \geq 1$ where $j \neq 0$, is not in $A_2$. Again we only need to consider the neighbors $y$ of a vertex $x$ with $\left \lVert x \right \rVert_k^{aj} \geq 1$ satisfying $\left \lVert y \right \rVert_k^{aj}=0$. Specifically, $y=x_1\cdots x_{n-2}ai$ with $i \neq j$, $y=x_1\cdots x_{n-2}ij$ with $i \neq a$ and $y=[0]^kx_{k+1}\cdots x_{n-2}aj$ if $\sum_{i=1}^k\mathbb{I}_{\{x_i \neq 0\}}=1$. The only neighbor $y$ of $x$ possibly in $A_1$ is the one satisfying $y=[0]^kx_{k+1}\cdots x_{n-2}aj$.Therefore, the vertex $x \notin A_2$. Similarly, the vertex $x$ with $\left \lVert x \right \rVert_k^{ib} \geq 1$ where $i \neq 0$, is not in $A_2$. 

Now let us turn to proving that the vertex $x$ with $\left \lVert x \right \rVert_k^{ij} \geq 1$, where $i \notin \{0,a\}$ and $j \notin \{0,b\}$, is not in $A_3$. Note that $x$ does not have any neighbor in $A_0$ so that $x \notin A_1$. Let us first show that $x \notin A_2$. The only neighbors $y$ of $x$ possibly in $A_1$ satisfy $\left \lVert y \right \rVert_k^{ij}=0$. Specifically, $y=x_1\cdots x_{n-2}aj$ with $a \neq i$, $y=x_1\cdots x_{n-2}ia$ with $a \neq j$ and $y=[0]^kx_{k+1}\cdots x_{n-2}ij$ if $\sum_{i=1}^k\mathbb{I}_{\{x_i \neq 0\}}=1$. The only neighbor $y$ of $x$ possibly in $A_1$ is the one satisfying $y=[0]^kx_{k+1}\cdots x_{n-2}ij$. Therefore, $x \notin A_2$. Note that the vertex $y$ with $y=[0]^kx_{k+1}\cdots x_{n-2}ij$ is not in $A_1$ since it does not have any neighbor in $A_0$. Very similarly, the only neighbors $y$ of $x$ possibly in $A_2$ satisfy $\left \lVert y \right \rVert_k^{ij}=0$. Specifically, $y=x_1\cdots x_{n-2}aj$ with $a \neq i$, $y=x_1\cdots x_{n-2}ia$ with $a \neq j$ and $y=[0]^kx_{k+1}\cdots x_{n-2}ij$ if $\sum_{i=1}^k\mathbb{I}_{\{x_i \neq 0 \}}=1$. The only neighbor $y$ of $x$ possibly in $A_2$ is the one satisfying $y=[0]^kx_{k+1}\cdots x_{n-2}ij$. Therefore, $x \notin A_3$.

Now we will move on to the inductive step. Now assume we have the following. 

For $t \geq 1$,
$$A_t \cap \{x \in \{0,1,\dots,q-1\}^n: \left \lVert x \right \rVert_k^{0b} \geq t\}=\emptyset,$$ 
and 
$$A_t \cap \{x \in \{0,1,\dots,q-1\}^n: \left \lVert x \right \rVert_k^{a0} \geq t\}=\emptyset.$$ 

For $j \notin \{b,0\}$ and $t \geq 1$,
$$A_{t+1} \cap \{x \in \{0,1,\dots,q-1\}^n: \left \lVert x \right \rVert_k^{0j} \geq t\}=\emptyset.$$

For $i \notin \{a,0\}$ and $t \geq 1$,
$$A_{t+1} \cap \{x \in \{0,1,\dots,q-1\}^n: \left \lVert x \right \rVert_k^{i0} \geq t\}=\emptyset.$$

For $j \neq 0$ and $t \geq 1$,
$$A_{t+1} \cap \{x \in \{0,1,\dots,q-1\}^n: \left \lVert x \right \rVert_k^{aj} \geq t\}=\emptyset.$$

For $i \neq 0$ and $t \geq 1$,
$$A_{t+1} \cap \{x \in \{0,1,\dots,q-1\}^n: \left \lVert x \right \rVert_k^{ib} \geq t\}=\emptyset.$$

For $i \notin \{a,0\}$, $j \notin \{b,0\}$, and $t\geq 1$,
$$A_{t+2} \cap \{x \in \{0,1,\dots,q-1\}^n: \left \lVert x \right \rVert_k^{ij} \geq t\}=\emptyset.$$

Let us consider a vertex $x$ satisfying $\left \lVert x \right \rVert_k^{0b} \geq t+1$. The neighbors $y$ of the vertex $x$ must satisfy 
$\left \lVert y \right \rVert_k^{0b} \geq t$, $\left \lVert y \right \rVert_k^{0j} \geq t+1$ for  $j \notin \{b,0\}$, $\left \lVert y \right \rVert_k^{00} \geq t+1$, $\left \lVert y \right \rVert_k^{ib} \geq t+1$ where $i \neq 0$ or $y=[0]^kx_{k+1}\cdots x_{n}$. From the induction hypothesis only two neighbors $y$ and $z$ of $x$ satisfying $\left \lVert y \right \rVert_k^{00} \geq t+1$ and $z=[0]^kx_{k+1}\cdots x_{n}$ can possibly be in $A_t$. However we claim that $z \notin A_t$.

Now let us prove the claim. In order for a such $z$ to exist, we must have 
$\sum_{i=1}^k\mathbb{I}_{\{x_i \neq 0\}}=1$ and $\sum_{i=k+1}^{k+l}\mathbb{I}_{\{x_i \neq 0\}}\geq t$. Let $z'=z_{k+1} \cdots z_{k+l}z_1\cdots z_k z_{k+l+1}z_{k+l+2}$ and then $\left \lVert z' \right \rVert_l^{0b} \geq t+1.$ Now let $S'=[\ast]^l[0]^{k+2}$ and $T'=[0]^l[\ast]^kab$. Then by the induction hypothesis, we have $z' \notin A_t^{'}$, where $A^{'}_t$ is the set of infected vertices with the initial infected set $S' \cup T'$. Thus $z \notin A_t$. Now we can conclude $x \notin A_{t+1}$. Similarly, the vertex $x$ satisfying $\left \lVert x \right \rVert_k^{a0} \geq t+1$ is not in $A_{t+1}$. 

Now consider a vertex $x$ satisfying $\left \lVert x \right \rVert_k^{0j} \geq t+1$ where $j \notin \{b,0\}$. The neighbors $y$ of the vertex $x$ must satisfy 
$\left \lVert y \right \rVert_k^{0i} \geq t+1$ for all $i \neq j$, $\left \lVert y \right \rVert_k^{0j} \geq t$, $\left \lVert y \right \rVert_k^{ij} \geq t+1$ for all $i \neq 0$, or $y=[0]^kx_{k+1}\cdots x_{n}$. By the induction hypothesis and what we have already proved, only two neighbors $y$ and $z$ of $x$ satisfying $\left \lVert y \right \rVert_k^{00} \geq t+1$ and $z=[0]^kx_{k+1}\cdots x_{n}$ can possibly be in $A_{t+1}$. However by using a very similar argument, we can show that that $z \notin A_{t+1}$. Therefore, $x \notin A_{t+2}$. In a very similar way, we can prove that $x \notin A_{t+2}$ if $x$ satisfies $\left \lVert x \right \rVert_k^{i0} \geq t+1$ for $i \notin \{a,0\}$,
$\left \lVert x \right \rVert_k^{aj} \geq t+1$ for $j \neq 0$, or $\left \lVert x \right \rVert_k^{ib} \geq t+1$ for $i \neq 0$.

Let us consider a vertex $x$ satisfying $\left \lVert x \right \rVert_k^{ij} \geq t+1$ where $i \notin \{a,0\}$ and $j \notin \{b,0\}$. We would like to show that $ x \notin A_{t+3}$. The neighbors $y$ of the vertex $x$ must satisfy $\left \lVert y \right \rVert_k^{ij} \geq t$, $\left \lVert y \right \rVert_k^{mj} \geq t+1$ where $m \neq i$, $\left \lVert y \right \rVert_k^{im} \geq t+1$ where $m \neq j$ or $y=[0]^kx_{k+1}\cdots x_{n}$. From the induction hypothesis and what we have already proved, the neighbor $y$ of $x$ satisfying $y=[0]^kx_{k+1}\cdots x_{n}$ is the only  possible vertex in $A_{t+2}$. Therefore, $x \notin A_{t+3}$. 
\end{proof}

Now let us summarize Lemmas (\ref{lemma:ST1})--(\ref{lemma:ST6}) in the following lemma. 
\begin{lemma} \label{lemma:summary}
For $x,y \in \{0,1,\cdots,q-1\}^n$ such that dim($x$)=$k$, dim($y$)=l, dim($x \vee y$)=$m$, where $k,l <m \leq n$, $d(x,y)=d \leq 2$, and $|\{i: x_i=y_i=* \}|=p$, then the infection time of $Q^x \cup Q^y$ in $Q_{n,q}$ is:

For $d=0$, $$T_{Q_{n,q}}(Q^x \cup Q^y)=m-p-1.$$

For $d=1$, $$T_{Q_{n,q}}(Q^x \cup Q^y)=m-p.$$

For $d=2$, $$T_{Q_{n,q}}(Q^x \cup Q^y)=m-p+1.$$
\end{lemma}
\begin{proof}
Note that we have $m=p+(k-p)+(l-p)+d=k+l-p+d$. 

Let us consider the case where $d=0$. By Lemma \ref{lemma:ST1}, we have 
$T_{Q_{m,q}}(Q^x \cup Q^y) \leq k-p+l-p-1=m-p-1$. By Lemma \ref{lemma:ST4}, we have $k-p+l-p-1=m-p-1$.

Now let us consider the case where $d=1$. By Lemma \ref{lemma:ST2}, we have 
$T_{Q_{m,q}}(Q^x \cup Q^y) \leq k-p+l-p+1=m-p$. By Lemma \ref{lemma:ST5}, we have $k-p+l-p+1=m-p$.

Finally let us consider the case where $d=2$. By Lemma \ref{lemma:ST3}, we have $T_{Q_{m,q}}(Q^x \cup Q^y) \leq k-p+l-p+3=m-p+1$. By Lemma \ref{lemma:ST6}, we have $k-p+l-p+3=m-p+1$.
\end{proof}

We are now in a position to prove the main theorem of this section, where we establish a recursive formula for the maximum percolation time, leading to its closed-form expression.

\begin{theorem}
For $q \geq 3$, $M_q(0)=0$ $M_q(1)=1$, and $M_q(2)=3$. Moreover,  for all $n \geq 3$
$$M_3(n)-M_3(n-3)=2n-1,$$
and for $q \geq 4$
$$M_q(n)-M_q(n-3)=2n.$$
\end{theorem}
\begin{proof}
It is obvious that $M_{q}(0)=0$ and $M_q(1)=1$. Now consider the case where $n=2$. W.l.o.g we can assume $00 \in A_0$ and first consider $A_0=\{00,ab\}$ where $a, b\neq 0$. It takes 3 steps to infect every vertex in $Q_q(2)$. Another initially infected set that we need to consider is $A_{0}=\{00,a0,0b\}$ where $a,b \neq 0$ and again it takes 3 steps to percolate. These are the only sets of initially infected vertices that we need to consider since adding any other vertex to these sets will not increase the percolation time and the process can not be initiated if any vertex is removed from these sets. 

We will first prove the lower bound of $M_{q}(n)$. 

Case 1. Let $A^{n-2}$ be a set that internally spans $Q^x_{n-2}$ for $x=[\ast]^{n-2}00$ in time $M_q(n-2)$ such that $[0]^n$ is infected at time $M_q(n-2)$. Now let $A'=A^{n-2} \cup [0]^{n-2}11$ and assume that we initially infect every vertex in $A'$. Note that $\langle A' \rangle=Q_{n,q}$ and before the time $M_q(n-2)$, none of the vertex in $[0]^{n-2}\ast \ast$ has interaction with the infection process in $[\ast]^{n-2}00$. 

Now assume that every vertex in $[\ast]^{n-2}00$ has been infected. Then it takes exactly $n+1$ additional steps to infect the rest of the vertices in $Q_{n,q}$. Indeed, by Lemma \ref{lemma:summary}, we have $d([\ast]^{n-2}00, [0]^{n-2}11)=2$, $|\{i:x_i=y_i=\ast\}|=0$ where $x=[\ast]^{n-2}00$ and $y=[0]^{n-2}11$, and $\langle [\ast]^{n-2}00, [0]^{n-2}11)=2\rangle= Q_{n,q}$. Therefore, 
$$T_{Q_{n,q}}(A')=M_q(n-2)+n+1.$$

Case 2. Let $A^{n-3}$ be a set that internally spans $Q_{n-3}^x$ for $x=[\ast]^{n-3}000$ in time $M_q(n-3)$ such that $[0]^n$ is infected at time $M_q(n-3)$. Let $A'=[\ast]^{n-3}000 \cup [0]^{n-3}110 \cup [2]^n$ and assume we initially infect every vertex in $A'$. Note that 
$\langle A' \rangle=Q_{n,q}$ and the set of sites infected after $M_q(n-3)$ steps is $[\ast]^{n-3}000 \cup [0]^{n-3}110 \cup [2]^n$.

After $M_q(n-3)+n-1$ steps, we claim that $[2]^n$ is the only vertex infected in $\cup_{j \neq 0} ([\ast]^{n-1}j)$. Indeed, the vertices $y$ in $[\ast]^{n-1}0$ that are at distance 2 from $[2]^n$ satisfy for $i \neq 2$ 
\begin{align} \label{eq:j20}
\left \lVert y \right \rVert^{i20}= n-3,  
\end{align} 

for $j \neq 2$
\begin{align}\label{eq:2j0}
    \left \lVert y \right \rVert^{2j0}= n-3,
\end{align}

 \begin{align} \label{eq:220}
\left \lVert y \right \rVert^{220}= n-4 
 \end{align}

or 
\begin{align} \label{eq:220'}
    \left \lVert y \right \rVert^{220}=n-3.
\end{align}

By Lemma \ref{lemma:ST3} and Lemma \ref{lemma:ST6}, none of those vertices $y$ satisfying (\ref{eq:j20})--(\ref{eq:220'}) is infected at time $ \leq M_q(n-3)+n-2$. In fact some of those vertices $y$ are infected exactly at time $M_q(n-3)+n-1$ and the rest of those vertices will be infected at time $M_q(n-3)+n$. Therefore, at time $M_q(n-3)+n-1$, $[2]^n$ is the only vertex infected in $\cup_{j \neq 0} ([\ast]^{n-1}j)$. Now we assume that every vertex in $[\ast]^{n-1}0 \cup [2]^n$ has been infected, then by Lemma \ref{lemma:summary}, it takes exactly $n$ steps to infected the entire $Q_{n,q}$ since $d([\ast]^{n-1}0,[2]^n)=1$, $|\{i:x_i =y_i=\ast\}|=0$ where $x=[\ast]^{n-1}0$ and $y=[2]^n$, and $\langle[\ast]^{n-1}0 \cup [2]^n \rangle=Q_{n,q}$.

Therefore, we have 
$$T_{Q_{n,q}}(A') \geq M_q(n-3)+2n-1.$$

Additionally we would like to show that for $q \geq 4$,
\begin{align}\label{eq:T(Qn)}
    T_{Q_{n,q}}(A') \geq M_q(n-3)+2n.
\end{align} 

In order to prove this tighter bound we need to further analyze the process right after $M_q(n-3)$ steps.

Now we need to show the following claim for $q \geq 4$. 
\begin{claim} \label{claim:subset I}
    For $ 1 \leq k \leq n-1$, every $k$ dimensional cube in $[\ast]^{n-1}2$ will be fully infected at time $M_q(n-3)+n+k$ and at time $\leq M_q(n-3)+n+k-1$ none of the $k$ dimensional cubes in $[\ast]^{n-1}2$ has been fully infected. In particular, the vertices $x$ with $x_n=2$ satisfying the following condition, denoted as Condition($k$),are not infected at time  $\leq M_q(n-3)+n+k-1$:
    
     there exists a subset $I \subset [n-1]$ of size at least $k$ such that 
\begin{align*}
x_i &=\begin{cases} 
  \notin \{0,2\}& \text{if $i \in I$ and $i \notin \{n-2,n-1\}$} \\
   \notin \{0,1,2\} & \text{if $i \in I$ and $ i \in \{n-2,n-1\}$} \\
   \in \{0,2\}& \text{if $i \notin I$ and $i \notin \{n-2,n-1\}$}\\
   \in \{0,1,2\} & \text{if $i \notin I$ and $i \in \{n-2,n-1\}$}
   \end{cases}
\end{align*}

    For $l \neq 0,2$ and  for $ 1 \leq k \leq n-1$, none of the $k$ dimensional cubes in $[\ast]^{n-1}l$ has been fully infected at time $ \leq M_q(n-3)+n+k$. In particular, the vertices $x$ with $x_n=l$ satisfying Condition($k$) are not infected at time $ \leq M_q(n-3)+n+k$.
    
\end{claim} 
\begin{proof}
Let $k=1$. Note that by Lemma \ref{lemma:ST3} and Lemma \ref{lemma:ST6}, after $M_q(n-3)+n-1$ steps every vertex $x$ in $[\ast]^{n-1}0$ has been infected except for the vertices $x$ satisfying 
$$\left \lVert x \right \rVert^{ij0}= n-3$$
where $i,j \notin \{0,1\}$
and after one additional step, the cube $[\ast]^{n-1}0$ has been fully infected. 

Again $[2]^n$ is the only vertex infected in $\cup_{j \neq 0} ([\ast]^{n-1}j)$ at time $M_q(n-3)+n-1$ and after one more step there are no infected vertices in $\cup_{j \notin \{0,2\}} ([\ast]^{n-1}j)$ besides $[2]^n$. The following vertices are infected in $[\ast]^{n-1}2$ at time $M_{q}(n-3)+n$:
  \begin{gather*}
 [2]^{n-3}122 \quad [2]^{n-3}022\\
[2]^{n-3}202 \quad [2]^{n-3}212
   \end{gather*}
and 
\begin{gather*}
0[2]^{n-1}\\
202^{n-2}\\
2202^{n-3}\\
...\\
[2]^{n-4}0[2]^3.
\end{gather*}

Indeed the vertex $[2]^{n-3}122$ has two neighbors infected at time $ \leq M_q(n-3)+n-1$, namely $[2]^n$ and $[2]^{n-3}120$ and it is similar for other vertices listed above.

At time $M_q(n-3)+n+1$, every $1$ dimensional cube in $[\ast]^{n-1}2$ is fully infected since the vertices in a $1$ dimensional cube are either infected at time $M_q(n-3)+n$ or having at least $2$ neighbors infected at time $M_q(n-3)+n+1$. 

Now consider the vertices in $\cup_{j \notin \{0,2\}} ([\ast]^{n-1}j)$. At time $M_q(n-3)+n+1$, the following vertices are infected:
$$\cup_{j \notin \{0,2\}}[2]^{n-1}j$$
since the vertex $[2]^{n-1}j$ has two neighbors infected at time $\leq M(n-3)+n$ namely $[2]^n$ and $[2]^{n-1}0$. Additionally, at time $M_q(n-3)+n+1$, for all $l \notin \{0,2\}$
  \begin{gather*}
 [2]^{n-3}12l \quad [2]^{n-3}02l\\
[2]^{n-3}20l \quad [2]^{n-3}21l
   \end{gather*}
and 
\begin{gather*}
0[2]^{n-2}l\\
202^{n-3}l\\
2202^{n-4}l\\
...\\
[2]^{n-4}0[2]^2l
\end{gather*}
are infected. It is easy to see that the rest of the vertices in $\cup_{j \notin \{0,2\}} ([\ast]^{n-1}j)$ are not infected at time $M_{q}(n-3)+n+1$.

Therefore, the statement is true for $k=1$.

Now assume the statement is true for $k \geq 1$. Consider a vertex $x$ with $x_n=2$ satisfying Condition($k+1$). By the induction hypothesis this vertex $x$ has only one neighbors at time $M_q(n-3)+n+k-1$, namely, $x_1x_2\cdots x_{n-1}0$. 

Indeed, let us choose a particular $i \in [n-1]$ where $i \notin \{n-1,n-2\}$ and $i \notin I$ and w.l.o.g we can assume that $x_i=0$. Consider a neighbor $y$ of $x$ where $y_j=x_j$ for $j \neq i$ and $y_i=2$. By the induction hypothesis, the vertex $y$ is not infected at time $M_q(n-3)+n+k-1$. Consider another neighbor $z$ of $x$ where $z_j=x_j$ for $j \neq i$ and $z_i=a$ where $a \notin \{0,2\}$. By the induction hypothesis, $z$ is not infected at time $M_q(n-3)+n-k-1$ since the vertex $z$ satisfies Condition($k+2$). The analysis is very similar for $i \in \{n-1,n-2\}$ and $i \notin I$.  

Now we choose a particular $i \in [n-1]$ where $i \notin \{n-1,n-2\}$ and $i \in I$. Consider a neighbor $y$ of $x$ where $y_j=x_j$ for $j \neq i$ and $y_i \notin \{x_i,0,2\}$. By the induction hypothesis $y$ is not infected at time $M_q(n-3)+n+k-1$ since the vertex $y$ satisfies Condition($k+1$). Consider another neighbor $z$ of $x$ with $z_j=x_j$ for $j \neq i$ and $z_i=a$ where $a \in \{0,2\}$. By the induction hypothesis the vertex $z$ is not infected at time $\leq M_q(n-3)+n+k-1$ since $z$ satisfies Condition($k$). The analysis is very similar for $i \in \{n-1,n-2\}$ and $i \notin I$.

Now consider a neighbor $y$ of $x$ where $y_j=x_j$ for $j \neq n$ and $y_n \notin \{0,2\}$.By the induction hypothesis, the vertex $y$ is not infected at time $M_q(n-3)+n-k$ since $y$ satisfies Condition($k$). 

Consider a vertex $x$ with $x_n=l$ satisfying Condition($k+1$). By the induction hypothesis this vertex $x$ has only one infected neighbor at time $M_q(n-3)+n+k$, namely, $x_1x_2\cdots x_{n-1}0$. The analysis is very similar to the case where $x_n=2$. 
\end{proof}
It is easy to see that claim \ref{claim:subset I} suffices to show (\ref{eq:T(Qn)}).

\vspace{5mm}
Now let us try to prove the upper bound of $M_{q}(n)$. Let $A$ be a set of initially infected vertices spanning the hypercube $Q_{n,q}$ by the process. 

Let  $Q_0=Q_{i_1}^{x_{i_1}} \subset Q_{i_2}^{x_{i_2}} \subset \cdots \subset Q_{i_t}^{x_{i_t}}=Q_{n,q}$ be a nested sequence of internally spanned subcubes with respect to $A$ and $Q_{m_1}^{Z_{m_1}}$,$Q_{m_2}^{Z_{m_2}}$,$\cdots$,$Q_{m_{t-2}}^{Z_{m_{t-2}}}$,$Q_{m_{t-1}}^{Z_{m_{t-1}}}$ be the cubes that marge with cubes $Q_{i_j}^{x_{i_j}}$ as described in Lemma \ref{lemma:nested sequence}. Assume w.l.o.g both sequences are of maximal lengths. 

W.l.o.g, assume that $A$ is minimal (under containment) set spanning $Q_{n,q}$. Note that $i_{t-1} \leq n$. We will proceed by analyzing different cases based on the the values of $i_{t-1}$ and $i_{t-2}$. By Lemma \ref{lemma:nested sequence}, for all $ 1 \leq j \leq t-1$ we have $i_j \geq m_j$

Case 1. If $i_{t-1} \leq n-2$, then after at most $M_q(i_{t-1}) \leq M_q(n-2)$ steps, both $Q_{i_{t-1}}^{x_{i_{t-1}}}$ and $Q_{m_{t-1}}^{z_{m_{t-1}}}$ are fully infected. Since $\langle Q_{i_{t-1}}^{x_{i_{t-1}}} \cup Q_{m_{t-1}}^{z_{m_{t-1}}} \rangle=Q_{n,q}$, after at most additional $n+1$ steps every vertex in $Q_{n,q}$ will be infected by Lemma \ref{lemma:summary}. Indeed, we have dim($Q_{i_{t-1}}^{x_{i_{t-1}}} \vee Q_{m_{t-1}}^{z_{m_{t-1}}}$)=n and consider the worst case where $d(x,y)=2$, and $|\{i:x_i=y_i=\ast \}|=0$ with $x=Q_{i_{t-1}}^{x_{i_{t-1}}}$ and $y=Q_{m_{t-1}}^{z_{m_{t-1}}}$.
Therefore, we have 
$$T_{Q_{n,q}}(A) \leq M(n-2)+n+1$$

Case 2. If $i_{t-1}=n-1$ and $i_{t-2}=n-2$, then there is a vertex $v \in A\cap Q_{m_{t-2}}^{z_{m_{t-2}}}$ such that $d(x_{t-2},v)=1$ and $\langle Q_{i_{t-2}}^{x_{i_{t-2}}} \cup v \rangle=Q_{i_{t-1}}^{x_{i_{t-1}}}.$ Moreover, there exists a vertex $w \in A \cap Q_{m_{t-1}}^{z_{m_{t-1}}}$ such that $\langle Q_{i_{t-1}}^{x_{i_{t-1}}} \cup w \rangle=Q_{n,q}.$ 

Note that $d(x_{i_{t-2}},w)=1$ or $d(x_{i_{t-2}},w)=2$. If $(x_{i_{t-2}},w)=2$, then $\langle Q_{i_{t-2}}^{x_{i_{t-2}}} \cup w \rangle=Q_{n,q},$ which contradicts the minimality of $A$ since $\langle A \backslash \{v\} \rangle =Q_{n,q}$. Thus $d(x_{i_{t-2}},w)=1$.

W.l.o.g we can assume that
$$x_{i_{t-2}}=[\ast]^{n-2}00,$$
$$x_{i_{t-2}} \vee v=[\ast]^{n-1}0,$$
and 
$$x_{i_{t-2}} \vee w=[\ast]^{n-2}0\ast.$$

It is easy to see that after at most $M_q(n-2)$ steps the cube $Q_{i_{t-2}}^{x_{i_{t-2}}}$ is fully infected. After at most $(n-1)$ additional steps both $Q_{n-1}^{x_{i_{t-2}} \vee v}$ and $Q_{n-1}^{x_{i_{t-2}} \vee w }$ are fully infected by Lemma \ref{lemma:summary}. After 2 more steps every vertex in $Q_{n,q}$ will be infected again by Lemma \ref{lemma:summary}. Therefore, we have
$$T_{Q_{n,q}}(A) \leq M_q(n-2)+n-1+2 =M_q(n-2)+n+1.$$

Case 3. If $i_{t-1}=n-1$, $i_{t-2} \leq n-3$, and $d({x_{i_{t-2}}},z_{m_{t-2}}) \leq 1$, then after $M(n-3)$ steps both $Q_{i_{t-2}}^{x_{i_{t-2}}}$ and $Q_{m_{t-2}}^{z_{m_{t-2}}}$ are fully infected. Since $\langle Q_{i_{t-2}}^{x_{i_{t-2}}} \cup Q_{m_{t-2}}^{z_{m_{t-2}}} \rangle=Q_{i_{t-1}}^{x_{i_{t-1}}}$, after at most $i_{t-1}$ more steps the cube $Q_{i_{t-1}}^{x_{i_{t-1}}}$ is fully infected by Lemma \ref{lemma:summary}. 

Since $i_{t-1}=n-1$, $d(x_{i_{t-1}},Z_{m_{t-1}}) \leq 1$. Again by Lemma \ref{lemma:summary} after at most $n$ more steps every vertex in $Q_{n,q}$ will be infected. Therefore, we have 
$$T_{Q_{n,q}}(A) \leq M_q(n-3)+n-1+n=M_q(n-3)+2n-1.$$

Case 4. If $i_{t-1}=n-1$,$i_{t-2} \leq n-3$, $d(x_{i_{t-2}},z_{m_{t-2}})=2$, and $m_{t-2}=n-3$, then after at most $M(n-3)$ steps both $Q_{i_{t-2}}^{x_{i_{t-2}}}$ and $Q_{m_{t-2}}^{z_{m_{t-2}}}$ are fully infected. Note that $\langle Q_{i_{t-2}}^{x_{i_{t-2}}} \cup Q_{m_{t-2}}^{z_{m_{t-2}}} \rangle=Q_{i_{t-1}}^{x_{i_{t-1}}}$ and since $i_{t-2} \geq m_{t-2}$ then $i_{t-2}=n-3$. After at most 3 additional steps the cube $Q_{i_{t-1}}^{x_{i_{t-1}}}$ is infected. Indeed observe that $|\{i: x_i=y_i=\ast\}| =n-3$ where $x= Q_{i_{t-2}}^{x_{i_{t-2}}}$ and $y= Q_{i_{t-2}}^{x_{i_{t-2}}}$ since one of the coordinates must be the same for $x$ and $y$. since $d(x,y)=2$ and $\text{dim}(x \vee y)=n-1$, $T_{Q_{n,q}}(Q_{i_{t-2}}^{x_{i_{t-2}}} \cup Q_{m_{t-2}}^{z_{m_{t-2}}})=3$ by Lemma \ref{lemma:summary}.

Since $i_{t-1}=n-1, d(x_{i_{t-1}},z_{m_{t-1}}) \leq 1$. By lemma \ref{lemma:summary} after at most $n$ more steps every vertex on $Q_{n,q}$ will be infected. Therefore, we have 
$$T_{Q_{n,q}}(A) \leq M_q(n-3)+3+n.$$

Additionally we would like to show 
$$T_{Q_{n,3}}(A) \leq M_3(n-3)+2+n.$$
W.l.o.g, we can assume that 
$$x_{i_{t-1}}=[\ast]^{n-1}0,$$
$$x_{i_{t-2}}=[\ast]^{n-3}000,$$
and 
$$z_{m_{t-2}}=[\ast]^{n-3}ij0,$$
where $i \neq 0$ and $j \neq 0$.

Again after $M_3(n-3)$ steps, both $Q_{i_{t-2}}^{x_{i_{t-2}}}$ and $Q_{m_{t-2}}^{z_{m_{t-2}}}$ are fully infected. After at most 2 additional steps every vertex in $Q_{i_{t-1}}^{x_{i_{t-1}}}$ satisfying $[\ast]^{n-3}kl0$ for all $k,l\in \{0,1,2\}$ but $k \notin \{0,i\}$ and $l \notin \{0,j\}$. 

There is a vertex $y \in A$ such that $A \in [\ast]^{n-1}l$ where $l \neq 0$. Otherwise the vertices not in the cube $Q_{i_{t-1}}^{x_{i_{t-1}}}$ will not be infected. Let $y=y_1y_2 \cdots y_{n-1}l$.

Therefore, after $M_3(n-3)+2$ steps, the vertices $v$ and $w$ in $Q_{i_{t-1}}^{x_{i_{t-1}}}$ are infected, where 
$$v=y_1y_2\cdots y_{n-3}\;i\;y_{n-1}0$$
and 
$$w=y_1y_2\cdots y_{n-3}\;0\; y_{n-1}0.$$
Assume $y_{n-2} \notin \{0,i\}$. Then after $M_3(n-3)+n$, vertices $v^*$ and $w^*$ in $[\ast]^{n-1}l$ are infected, where 
$$v^*=y_iy_2 \cdots y_{n-3}\;i\;y_{n-1}l$$
and 
$$w^*=y_iy_2 \cdots y_{n-3}\;0\;y_{n-1}l. $$
Therefore, after $M_3(n-3)+3$, $[\ast]^{n-1}0$ and $y_1y_2\cdots y_{n-3}\ast y_{n-1}l$ have been infected. By Lemma \ref{lemma:summary} after at most $n-1$ steps every vertex in $Q_{n,3}$ will be infected. 

Now assume $y_{n-2}=0$. If $y_{n-1} \notin \{0,j\}$, similarly it can be shown that after $M_3(n-3)+n+2$ every vertex in $Q_{n,3}$ will be infected. Thus we can assume $y_{n-1} \in \{0,j\}$. First let $y_{n-1}=0$. Note that after $M_3(n-3)+n-1$ steps, the vertices $u$ and $z$ in $Q_{i_{t-1}}^{x_{i_{t-1}}}$ are infected where 
$$u=y_1y_2 \cdots y_{n-3}k00$$
with $k \notin\{i,0\}$
and 
$$z=y_1y_2 \cdots y_{n-3}i00.$$
Then after one more step, the vertices $u^*$ and $z^*$ in $[\ast]^{n-1}l$ are infected, where 
$$u^*=y_1y_2 \cdots y_{n-3}k0l$$
and 
$$z^*=y_1y_2 \cdots y_{n-3}i0l.$$
Therefore, after $M_3(n-3)+3$ steps, $[\ast]^{n-1}0$ and $y_1y_2\cdots y_{n-3}\ast 0l$ have been infected. By Lemma \ref{lemma:summary} after at most $n-1$ steps every vertex in $Q_{n,3}$ will be infected. 

Due to symmetry for other values of $y_{n-2}$ and $y_{n-1}$ the argument is the same. 

Therefore, we have 
$$T_{Q_{n,3}} \leq M_3(n-3)+n+2.$$

Case 5. If $i_{t-1}=n-1$,$i_{t-2} \leq n-3$, $d(x_{i_{t-2}},z_{m_{t-2}})=2$, and $m_{t-2} < n-3$, after at most $M_q(n-3)$ steps, both $Q_{i_{t-2}}^{x_{i_{t-2}}}$ and $Q_{m_{t-2}}^{z_{m_{t-2}}}$ are fully infected. After at most $n$ more steps the cube $Q_{i_{t-1}}^{x_{i_{t-1}}}$ is fully infected by Lemma \ref{lemma:summary}. Then after at most $n$ more 
steps every vertex on $Q_{n,q}$ is infected again by Lemma ~\ref{lemma:summary} since $d(x_{i_{t-2}},z_{m_{t-2}}) \leq 1$ and $\langle Q_{i_{t-1}}^{x_{i_{t-1}}} \cup Q_{m_{t-1}}^{z_{m_{t-1}}} \rangle=Q_{n,q}.$ Therefore, we have for $q \geq 3$
$$T_{Q_{n,q}}(A) \leq M_q(n-3)+n+n=M_q(n-3)+2n.$$

Additionally we would like to show that 
$$T_{Q_{n,3}}(A) \leq M_3(n-3)+2n-1.$$

In order to show this tighter bound we need to analyze the process after $M_3(n-3)$ steps in a more careful way. Assume that $q=3$. 

W.l.o.g we can assume that 
$$x_{i_{t-1}}=[\ast]^{n-1}0,$$
$$x_{i_{t-2}}=[\ast]^{i_{t-2}}[0]^{n-i_{t-2}},$$
and 
$$z_{m_{t-2}}=[\ast]^p[0]^{i_{t-2}-p}[\ast]^{m_{t-2}-p}ij0,$$
where $i \neq 0$ and $j \neq 0$.

Again after at most $M(n-3)$ steps both $Q_{i_{t-2}}^{x_{i_{t-2}}}$ and $Q_{m_{t-2}}^{z_{m_{t-2}}}$ are fully infected. After at most $n-1$ additional steps, by Lemma \ref{lemma:ST3} and Lemma \ref{lemma:ST6}, vertices $x$ in $Q_{i_{t-1}}^{x_{i_{t-1}}}$ satisfying $ 1 \leq \left \lVert x \right \rVert^{kl0} \leq n-3$  for all $k,l\in \{0,1,2\}$ but $k \notin \{0,i\}$ and $l \notin \{0,j\}$ are infected.

There is a vertex $y \in A$ such that $A \in [\ast]^{n-1}l$ where $l \neq 0$. Otherwise the vertices not in the cube $Q_{i_{t-1}}^{x_{i_{t-1}}}$ will not be infected. Let $y=y_1y_2 \cdots y_{n-1}l$. 

Assume that $\sum_{i=1}^{n-3}\mathbb{I}_{\{y_i \neq 0\}} \geq 1$. Therefore, after $M_3(n-3)+n-1$ steps, the vertices $v$ and $w$ in $Q_{i_{t-1}}^{x_{i_{t-1}}}$ are infected, where 
$$v=y_1y_2\cdots y_{n-3}\;i\;y_{n-1}0$$
and 
$$w=y_1y_2\cdots y_{n-3}\;0\; y_{n-1}0.$$
Assume $y_{n-2} \notin \{0,i\}$. Then after $M_3(n-3)+n$, vertices $v^*$ and $w^*$ in $[\ast]^{n-1}l$ are infected, where 
$$v^*=y_iy_2 \cdots y_{n-3}\;i\;y_{n-1}l$$
and 
$$w^*=y_iy_2 \cdots y_{n-3}\;0\;y_{n-1}l. $$
Therefore, after $M(n-3)+n$, $[\ast]^{n-1}0$ and $y_1y_2\cdots y_{n-3}\ast y_{n-1}l$ have been infected. By Lemma \ref{lemma:summary} after at most $n-1$ steps every vertex in $Q_{n,3}$ will be infected. 

Now assume $y_{n-2}=0$. If $y_{n-1} \notin \{0,j\}$, similarly it can be shown that after $M_3(n-3)+2n-1$ every vertex in $Q_{n,3}$ will be infected. Thus we can assume $y_{n-1} \in \{0,j\}$. First let $y_{n-1}=0$. Note that after $M_3(n-3)+n-1$ steps, the vertices $u$ and $z$ in $Q_{i_{t-1}}^{x_{i_{t-1}}}$ are infected where 
$$u=y_1y_2 \cdots y_{n-3}k00$$
with $k \notin\{i,0\}$
and 
$$z=y_1y_2 \cdots y_{n-3}i00.$$
Then after one more step, the vertices $u^*$ and $z^*$ in $[\ast]^{n-1}l$ are infected, where 
$$u^*=y_1y_2 \cdots y_{n-3}k0l$$
and 
$$z^*=y_1y_2 \cdots y_{n-3}i0l.$$
Therefore, after $M_3(n-3)+n$ steps, $[\ast]^{n-1}0$ and $y_1y_2\cdots y_{n-3}\ast 0l$ have been infected. By Lemma \ref{lemma:summary} after at most $n-1$ steps every vertex in $Q_{n,3}$ will be infected. 

Due to symmetry for other values of $y_{n-2}$ and $y_{n-1}$ the argument is the same. 

Therefore, we have 
$$T_{Q_{n,3}} \leq M_3(n-3)+2n-1.$$

Assume $y=[0]^{n-3}y_{n-2}y_{n-1}l$. Therefore, after $M_3(n-3)+n-1$ steps, as long as $n-1 \geq 2$, the vertices $v$ and $w$ in $Q_{i_{t-1}}^{x_{i_{t-1}}}$ are infected, where 
$$v=[0]^{n-3}\;i\;y_{n-1}0$$
and 
$$w=[0]^{n-3}\;0\; y_{n-1}0.$$
Indeed, the vertex $[0]^n$ and $[0]^{n-3}ij0$ are infected at time $\leq M_3(n-3)$. From now on the proof is identical to the case where $\sum_{i=1}^{n-3}\mathbb{I}_{\{y_i \neq 0\}} \geq 1$.



Finally we have 
$$T_{Q_{n,3}} \leq M_3(n-3)+2n-1.$$

We need one more claim to conclude the proof. 
\begin{claim}
If $M_q(0)=0$,$M_q(1)=1$ and $M_q(2)=3$ and  
for $n \geq 3$ 
$$M_3(n)=\max\{M_3(n-2)+n+1,M_3(n-3)+2n-1\}$$
and for $q \geq 4$
$$M_q=\max\{M_q(n-2)+n+1,M_q(n-3)+2n\},$$ then 
$$M_3(n)=M_3(n-3)+2n-1$$
and for $q \geq 4$
$$M_q(n)=M_q(n-3)+2n.$$
\end{claim}
\begin{proof}
Assume that $q=3$. We have $M_3(3)=M_3(0)+5=5$ or $M_3(3)=M_3(1)+4=5$

$M_3(4)=M_3(1)+7=8$ or 
$M_3(4)=M_3(2)+5=8$.

$M_3(5)=M_3(2)+9=12$ or 
$M_3(5)=M_3(3)+6=11$.

Thus the lemma holds for $ 1 \leq n \leq 3$.

For $n \geq 3$, we assume that the lemma holds for $n,n+1,n+2$. We have 
\begin{align}
    M_3(n-3)& =\max\{M_3(n)+2(n+3)-1,M_3(n+1)+n+4\}\\
    &=\max\{M_3(n-3)+2n-1+2(n+3)-1, M_3(n-2)+2(n+1)-3+n+4\}\\
    &= \max{M_3(n-3)+4n+4,M_3(n-2)+3n+3}\\
   \label{eq:M(n-3+4n+4)} &=M_3(n-3)+4n+4\\ 
    &=M_3(n)-2n+1+4n+4\\
    &=M_3(n)+2n+5
\end{align}
where (\ref{eq:M(n-3+4n+4)}) follows from 
\begin{align*}
    M_3(n-3)+4n+4& =M_3(n-3)-2n+1+2n+5\\
    & \geq M_3(n-2)+n+1+2n+5\\
    & > M_3(n-2)+3n+3.
\end{align*}
For $q \geq 4$, the calculation is almost the same. 
\end{proof}
Therefore, we have our desired results.
\end{proof} 
Now we are ready to prove Theorem \ref{Theorem1}. 

\begin{proof}
Since $M_3(n)=M_3(n-3)+2n-1$, we have $M_3(n)=\frac{n^2}{3}+\frac{2n}{3}+C$ where $C$ is an undetermined real constant. Since $M_3(0)=0$, we have $C=0$ and thus $M_3(n)=\frac{n^2}{3}+\frac{2n}{3}$ for $n=3,6,9,\cdots$.

Other expressions in Theorem \ref{Theorem1} can be derived in the same way. 
\end{proof}
\section{Open problems}
We conclude with several open problems. 
\begin{enumerate}
    \item What is the maximum percolation time on the $n$-dimensional $q$-ary hypercube when the infection threshold $r \geq 3$?
    
\noindent The value $M(Q_{n,2},r \geq 3)=\frac{2^n}{n}(\log n)^{-O(1)}$ was obtained in \cite{Ivailo} by discovering a link between this problem and a generalization of the Snake-in-the-Box problem. Improving upon this result to achieve an asymptotically tight, up to a constant, is an interesting question. Finding an exact expression for $M(Q_{n,q},r \geq 3)$ remains as a challenging open problem. 

\item What is the maximum percolation time, denoted by $M([n]^d,r)$, on the $[n]^d$ grid when the infection threshold $ 2 \leq r \leq 2d$?

\noindent The value $M([n]^2,2)=\frac{3n^2}{8}$ was established in \cite{Benevides}. However, for the other values of $d$ and $r$, the behavior of $M([n]^d,r)$ has not been explored. 
\end{enumerate}

\section*{Acknowledgment}
The author would like to thank Prof.~Alexander Barg for suggesting this problem and providing valuable guidance. Additionally, the author extends thanks to  Micha{\l} Przykucki for addressing questions related to \cite{Michal}. 

\bibliographystyle{plain}
\bibliography{main.bib}
\end{document}